\def\draft{n}
\documentclass{amsart}
\usepackage[headings]{fullpage}
\usepackage{amssymb,epic,eepic,epsfig,amsmath,amsbsy,amsmath,bbm}
\usepackage{graphicx}
\usepackage{texdraw}
\usepackage{url}
\usepackage[bookmarks=true,%
    colorlinks=true,%
    linkcolor=blue,%
    citecolor=blue,%
    filecolor=blue,%
    menucolor=blue,%
    urlcolor=blue,%
    breaklinks=true]{hyperref}

\newtheorem{theorem}{Theorem}[section]
\newtheorem{proposition}{Proposition}[section]
\theoremstyle{definition}
\newtheorem{lemma}[proposition]{Lemma}
\newtheorem{definition}[proposition]{Definition}
\newtheorem{remark}[proposition]{Remark}
\newtheorem{corollary}[proposition]{Corollary}

\def\printname#1{
        \if\draft y
                \smash{\makebox[0pt]{\hspace{-0.5in}
                        \raisebox{8pt}{\tt\tiny #1}}}
        \fi
}

\newcommand{\psdraw}[2]
         {\begin{array}{c} \hspace{-1.3mm}
        \raisebox{-4pt}{\epsfig{figure=draws/#1.eps,width=#2}}
        \hspace{-1.9mm}\end{array}}

\newlength{\standardunitlength}
\catcode`\@=11
\long\def\@makecaption#1#2{%
     \vskip 10pt

\setbox\@tempboxa\hbox{
       \small\sf{\bfcaptionfont #1. }\ignorespaces #2}%
     \ifdim \wd\@tempboxa >\captionwidth {%
         \rightskip=\@captionmargin\leftskip=\@captionmargin
         \unhbox\@tempboxa\par}%
       \else
         \hbox to\hsize{\hfil\box\@tempboxa\hfil}%
     \fi}
\font\bfcaptionfont=cmssbx10 scaled \magstephalf
\newdimen\@captionmargin\@captionmargin=2\parindent
\newdimen\captionwidth\captionwidth=\hsize
\catcode`\@=12

\def\lbl#1{\label{#1}\printname{#1}}

\def\BN{\mathbb N}
\def\BZ{\mathbb Z}

\def\BQ{\mathbb Q}
\def\BR{\mathbb R}
\def\BC{\mathbb C}

\def\calT{\mathcal T}

\def\a{\alpha}

\def\l{\lambda}

\def\ga{\gamma}

\def\la{\langle}
\def\ra{\rangle}

\def\e{\epsilon}

\def\d{\delta}

\def\longto{\longrightarrow}

\def\SL{\mathrm{SL}}
\def\pt{\partial}

\def\calN{\mathcal{N}}
\def\calA{\mathcal{A}}

\def\d{\delta}
\def\lt{\mathrm{lt}}
\def\calD{\mathcal{D}}
\def\ch{\mathrm{ch}}

\begin{document}


\title[The degree of a $q$-holonomic sequence is a quadratic quasi-polynomial]{
The degree of a $q$-holonomic sequence is a quadratic quasi-polynomial}
\author{Stavros Garoufalidis}
\address{School of Mathematics \\
         Georgia Institute of Technology \\
         Atlanta, GA 30332-0160, USA \newline 
         {\tt \url{http://www.math.gatech.edu/~stavros}}}
\email{stavros@math.gatech.edu}
\thanks{The author was supported in part by NSF. \\
\newline
1991 {\em Mathematics Classification.} Primary 57N10. Secondary 57M25.
}

\date{March 2, 2011}


\dedicatory{To Doron Zeilberger, on the occasion of his 60th birthday}

\begin{abstract}
A sequence of rational functions in a variable $q$ is $q$-holonomic if it
satisfies a linear recursion with coefficients polynomials in $q$ and $q^n$.
We prove that the degree of a $q$-holonomic sequence is eventually 
a quadratic quasi-polynomial, and that the leading term satisfies a linear
recursion relation with constant coefficients.
 Our proof uses differential Galois theory 
(adapting proofs regarding holonomic $D$-modules to the case of $q$-holonomic 
$D$-modules) combined with the Lech-Mahler-Skolem theorem from number theory.
En route, we use the Newton polygon
of a linear $q$-difference equation, and introduce the notion of 
regular-singular $q$-difference equation and a WKB basis of solutions
of a linear $q$-difference equation at $q=0$. We then use the 
Skolem-Mahler-Lech theorem to study the vanishing of their leading term. 
Unlike the case of $q=1$, there are no analytic problems regarding 
convergence of the WKB solutions. Our proofs are constructive, and they
are illustrated by an explicit example.
\end{abstract}

\maketitle

\tableofcontents

\section{Introduction}
\lbl{sec.intro}

\subsection{History}
\lbl{sub.history}

$q$-holonomic sequences appear in abundance in Enumerative Combinatorics;
\cite{PWZ,St}. Here and and below, $q$ is a variable, and not a complex
number. The fundamental theorem of Wilf-Zeilberger states that 
a multi-dimensional finite sum of a (proper) $q$-hyper-geometric term
is always $q$-holonomic; see \cite{WZ,Z,PWZ}. Given this result,
one can easily generate $q$-holonomic sequences. We learnt about this
astonishing result from Doron Zeilberger in 2002. Putting this
together with the fact that many {\em state-sum} invariants in Quantum 
Topology are multi-dimensional sums of the above shape, it follows
that Quantum Topology provides us with a plethora of $q$-holonomic
sequences of {\em natural origin}; \cite{GL}. For example, the sequence
of {\em Jones polynomials} of a knot and its parallels (technically,
the colored Jones function) is $q$-holonomic. Moreover, the corresponding
minimal recursion relation can be chosen canonically and is
conjecturally related to geometric invariants of the knot; see \cite{Ga1}.
Recently, the author focused on the degree of a $q$-holonomic sequence,
and in the case of the Jones polynomial it is also conjecturally related to 
topological invariants of knot complement; see \cite{Ga2}. Since 
little is known about the Jones polynomial of a knot and its parallels,
one might expect no regularity on its sequence of degree. The contrary
is true, and in fact is a property of $q$-holonomic sequences and 
the focus of this paper. Our results were announced in \cite{Ga2} and 
\cite{Ga3}, where numerous examples of geometric/topological origin were 
discussed.

\subsection{The degree and the leading term of a $q$-holonomic sequence}
\lbl{sub.cj}

Our main  theorem concerns the degree and the leading term of a $q$-holonomic
sequence. To phrase it, we need to recall what is a $q$-holonomic
sequence, and what is a quasi-polynomial.
A sequence $(f_n(q))$ of rational functions is $q$-{\em holonomic} if it
satisfies a recursion relation of the form
\begin{equation}
\lbl{eq.rec}
a_d(q^n,q) f_{n+d}(q) + \dots + a_0(q^n,q) f_{n}(q)=0
\end{equation}
for all $n$ where $a_j(u,v) \in \BQ[u,v]$ for $j=0,\dots,d$
and $a_d(u,v) \neq 0$. 

Given a polynomial with rational coefficients
$$
f(q)=\sum_m^M c_i q^i \in \BQ[q]
$$
that satisfies $c_m c_M \neq 0$, its {\em degree} (also known as the
{\em order} in different contexts) $\d(f(q))$ and its 
{\em leading term} $\lt(f(q))$ is given by $m$ and $c_m$ respectively.
In other words, the degree and the leading term of $f(q)$ is the order and
the starting coefficient in the Taylor series expansion of $f(q)$ at $q=0$.
The degree and leading term  
can be uniquely extended to the field $\BQ(q)$ of all rational functions, 
the ring $\BQ[[q]]$ of formal power series in $q$, the field $\BQ((q))$
of all Laurent series in $q$ and finally to the 
{\em algebraically closed field} 
\begin{equation}
\lbl{eq.K}
K=\overline{\BQ}\{\{q\}\}=\cup_{r=1}^\infty \overline{\BQ}((q^{1/r}))
\end{equation} 
of all {\em Puiseux series} in $q$ with algebraic coefficients 
(see \cite{Wa}). The field $K$ is required in Theorems \ref{thm.2} 
and \ref{thm.3} below.

Recall that a {\em quasi-polynomial} $p(n)$ is a 
function 
$$
p: \BN \longto \BN, \qquad p(n)=\sum_{j=0}^d c_j(n) n^j
$$ 
for some $d \in \BN$ where $c_j(n)$ is a periodic fuction with integral 
period for $j=1,\dots,d$; \cite{St,BR}.
If $c_j=0$ for $j>2$, then we will say that $p$ is a 
{\em quadratic} quasi-polynomial. We will say that a function is eventually
equal to a quasi-polynomial if they agree for all but finitely many
values.

\begin{theorem}
\lbl{thm.1}
\rm{(a)}
The degree of a $q$-holonomic sequence is eventually a quadratic 
quasi-polynomial.
\newline
\rm{(b)}
The leading term of a $q$-holonomic sequence eventually 
satisfies a linear recursion with {\em constant} coefficients.
\end{theorem}
Theorem \ref{thm.1} follows from Theorems \ref{thm.2} and
\ref{thm.4} below.

\begin{remark}
\lbl{rem.othercoeff}
If $f_n(q) \in \BZ[q^{\pm 1}]$ is $q$-holonomic with degree $\d_n$ and leading
term $\lt_n$, it follows that $f_n(q)-\lt_n q^{\d_n}$ is also $q$-holonomic.
Thus, Theorem \ref{thm.1} applies to each one of the terms of $f_n(q)$. 
\end{remark}

\begin{remark}
\lbl{rem.divisio}
L. Di Vizio brought to our attention that results similar to part (a) of 
Theorem \ref{thm.1} appear in \cite[Thm.4.1]{BB} and also in 
\cite[Sec.1.1,Sec.1.3]{DV}. However, the statement and proof of Theorem
\ref{thm.1} appear to be new. 
\end{remark}

Our next corollary to Theorem \ref{thm.1} characterizes which sequences
of monomials are $q$-holonomic sequences. In a sense, such sequences
are building blocks of all $q$-holonomic sequences.

\begin{corollary}
\lbl{cor.thm1}
If $(a_n)$ and $(b_n)$ are sequences of integers (with $a_n \neq 0$
for all $n$), then $(a_n q^{b_n})$ is $q$-holonomic
if and only if $b_n$ is holonomic and $a_n$ is a quadratic quasi-polynomial
for all but finitely many $n$.
\end{corollary}

\subsection{The Newton polygon of a linear $q$-difference equation}
\lbl{sub.newton}

As is common, we can write a linear $q$-difference equation \eqref{eq.rec}
in operator form by introducing two operators $L$ and $M$ which
act on a sequence $(f_n(q))$ by
$$
(Lf)_n(q)=f_{n+1}(q), \qquad (Mf)_n(q)=q^n f_n(q).
$$
It is easy to see that the operators $M$ and $L$ $q$-commute
\begin{equation}
\lbl{eq.qML}
LM=q ML
\end{equation}
and generate the so-called $q$-{\em Weyl algebra} 
\begin{equation}
\lbl{eq.qweyl}
\calD=\cup_{r=1}^\infty K((M^{1/r}))\la L\ra/(LM^{1/r}-q^{1/r}ML).
\end{equation}
We will call an element of $\calD$ a {\em linear $q$-difference operator}. 
The general element of $\calD$ is of the form
\begin{equation}
\lbl{eq.recP}
P=\sum_{i=0}^d a_i(M,q) L^i
\end{equation}
where $a_i(M,q) \in K((M^{1/r}))$ for $i=0,\dots,d$ and
$a_d(M,q) \neq 0$. The equation 
$$
Pf =0
$$ 
for a $K$-valued sequence $f_n(q)$ is exactly the recursion relation 
\eqref{eq.rec}. The {\em Newton polygon} $N(P)$ of an element $P \in \calD$
is defined to be the {\em lower convex hull} of the points $(i,\d_M(a_i))$
where $\d_M$ denotes the smallest degree with respect to $M$. 
Note that usually the Newton polygon $N'(P)$ of a 2-variable polynomial is
defined to be the convex hull of the exponents of its monomials. Since
we are working locally at $q=0$, we view $N'(P)$ by placing our eye at
$-\infty$ in the vertical axis and looking up. The resulting object
is the lower convex hull $N(P)$ defined above. The Newton polygon of
a linear $q$-difference equation was also studied in the recent Ph.D.
thesis by P. Horn; see \cite[Chpt.2]{Horn}.

The Newton polygon $N(P)$ of $P$ is a finite union of intervals 
with rational end points and two vertical rays. Each interval with end-points
$(i,d_i)$ and $(j,d_j)$ for $i<j$ has a {\em slope}
$s=(d_j-d_i)/(j-i)$ and a {\em length} (or {\em multiplicity}) $l=j-i$. 
The {\em multiset of slopes} $s(N(P))$ of $N(P)$ is the set of slopes
of $s$ $N(P)$ each with multiplicity $l_s$. An example of a Newton polygon 
of a linear $q$-difference operator of degree $7$ with three slopes 
$-1,0,1/2$ of multiplicity $2,1,4$ respectively is shown
here:
$$
\psdraw{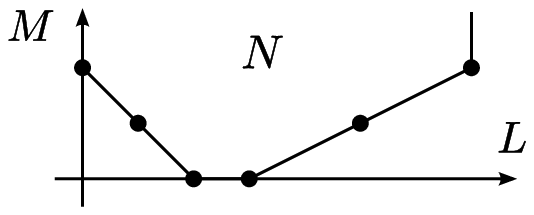}{2in}
$$ 
The Newton polygon is a convenient way to organize solutions to a linear
$q$-difference equation. The reader may compare this with the
Newton polygon of a linear differential operator attributed to Malgrange
and Ramis; see \cite[Sec.3.3]{vPS}
and references therein. 
In analogy with the theory of linear differential operators, we will
say that $P$ is a {\em regular-singular} $q$-difference operator if
its Newton polygon consists of a single horizontal segment. In other words,
after a minor change of variables, with the notation of \eqref{eq.recP}
this means that $a_i(M,q) \in K[[M^{1/r}]]$ for all $i$ and 
$a_0(0,q) a_d(0,q) \neq 0$.

\subsection{WKB sums}
\lbl{sub.WKB}

Since Equation \eqref{eq.rec} involves two independent variables $q$ and $q^n$,
let us set $q^n=u$ and consider the $q$-difference equation 
\begin{equation}
\lbl{eq.rec2}
a_d(u,q) f(uq^d,q) + \dots + a_0(u,q) f(u,q)=0
\end{equation}
Note that if $f(u,q)$ solves \eqref{eq.rec2}, and if $f_n(q):=f(q^n,q)$
makes sense, then $f_n(q)$ solves \eqref{eq.rec} and vice-versa. 
It turns out that the general solution to 
\eqref{eq.rec2} is a WKB sum. The latter are generalizations of 
the better known generalized power sums discussed in Section \ref{sub.gps}
below.

\begin{definition}
\lbl{def.WKB}
\rm{(a)}
A {\em formal WKB series} is an expression of the form
\begin{equation}
\lbl{eq.WKB}
f_{\tau}(u,q)= q^{\ga n^2} \l(q)^n A(n,u,q) 
\end{equation}
where $\tau=(\ga,\l(q),A)$ and 
\begin{itemize}
\item
the exponent $\ga$ is a rational number,
\item
$$
A(n,u,q)=\sum_{i=0}^M \sum_{k=0}^\infty \phi_{i,k}(q) u^{k/r} n^i
$$ 
is a polynomial in $n$ with coefficients in 
$\overline{\BQ}((q^{1/r}))[[u^{1/r}]]$ for some $r \in \BN$.
\item
$\l(q) \in \overline{\BQ}((q^{1/r}))$ and $\phi_{0,0}(q)=1$.
\item
there exists $c \in \BQ$ so that 
$\d(\phi_{i,k}(q)) \geq ck $ for all $i$ and $k$
\end{itemize}
\rm{(b)}
A {\em formal WKB sum} is a finite $K$-linear combination
of formal WKB series.
\end{definition}

\noindent
Observe that if $f_{\tau}(u,q)$ is a formal WKB series, then its evaluation
\begin{equation}
\lbl{eq.evalf}
f_{\tau,n}(q)=f_{\tau}(q^n,q) \in K
\end{equation}
is a well-defined $K$-valued sequence for $n > -r c$.

\noindent
Observe further if $f_{\tau}(u,q)$ is given by \eqref{eq.WKB}, 
the operators $L$ and $M$ act on $f_{\tau}(u,q)$ 
by:
\begin{equation}
\lbl{eq.LMf}
(Lf_{\tau})(u,q)=q^{\ga (n+1)^2}\l(q)^{n+1} A(n+1,uq,q)
, \qquad (Mf_{\tau})(u,q)= q^{\ga n^2}\l(q)^n u A(n,u,q).
\end{equation}

\begin{theorem}
\lbl{thm.2}
Every linear $q$-difference equation has a basis of solutions of the
form $f_{\tau,n}(q)$. Moreover, the multiset of exponents of the bases is
the multiset of the negatives of the slopes of the Newton polygon.
\end{theorem}

Recall that $K\{\{M\}\}$ denotes the field of Puiseux series in a variable
$M$ with coefficients in $K$. 
Let $\d_M(f)$ denote the minimum exponent of $M$ in a Puiseux series
$f \in K\{\{M\}\}$.

\begin{theorem}
\lbl{thm.3}
Every monic $q$-difference operator $P \in \calD$ of order $d$ can be factored
as an ordered product
\begin{equation}
\lbl{eq.Pfactor}
P=\prod_{i=1}^d (L-a_i(M,q))
\end{equation}
where $a_i(M,q) \in K\{\{M\}\}$ and the multiset $\{\d_M(a_i)|i=1,\dots,d\}$
of slopes is the negative of the multiset of slopes of the Newton polygon
of $P$.
\end{theorem}

\begin{remark}
\lbl{rem.fivebranes}
The WKB expansion of solution of linear $q$-difference equations given in
Theorem \ref{thm.2} appears to be new, and perhaps it is related to some
recent work of Witten \cite{Witten}, who proposes a categorification
of the colored Jones polynomial in an arbitrary 3-manifold. The curious
reader may compare \cite[Eqn.6.21]{Witten} with our Theorem \ref{thm.2}.
We wish to thank T. Dimofte for pointing out the reference to us.
\end{remark}

\begin{remark}
\lbl{rem.noq}
Although the proofs of 
Theorems \ref{thm.2} and \ref{thm.3} follow the well-studied case of linear
differential operators in one variable $x$, we are unable to formulate
an analogue of Theorems \ref{thm.2} and \ref{thm.3}
in the differential operator case. Indeed, $q$ is a variable which seems to
be independent of the spacial variables $x$ of a differential operator. 
The meaning
of the variable $q$ can be explained by {\em Quantization}, or from
{\em Tropical Geometry} (where it is usually denoted by $t$) or from
the representation theory of {\em Quantum Groups}; see for example
\cite{Ga3} and referencies therein, and also Section 
\ref{sec.invariants} below. 
\end{remark}

\begin{theorem}
\lbl{thm.4}
If $f_n(q)$ is a finite $K$-linear combination of formal
WKB series of the form \eqref{eq.WKB}, then for large $n$
its degree is given by a quadratic quasi-polynomial and its leading
term satisfies a linear recursion relation with constant coefficients.
\end{theorem}

\subsection{An example}
\lbl{sub.example}

In this section we discuss in detail an example to illustrate the 
introduced notions and the content of Theorems \ref{thm.1}, \ref{thm.2} 
and \ref{thm.4}. The example shows that the proof of Theorem \ref{thm.1} and
the WKB expansion of Theorem \ref{thm.4} is algorithmic, despite the
use of differential Galois theory. Consider the 
$q$-holonomic sequence $f_n(q) \in \BZ[q]$ whose first few terms are given by:

\small{
\begin{eqnarray*}
f_0 &=& 
1 
\\ f_1 &=& 
2 - q^{2} 
\\ f_2 &=& 
3 + q - 2 q^{3} - 2 q^{4} + q^{7} 
\\ f_3 &=& 
4 + 2 q + 2 q^{2} - 3 q^{4} - 
 4 q^{5} - 4 q^{6} - q^{7} + 2 q^{9} + 2 q^{10} + 2 q^{11} - q^{15} 
\\ f_4 &=& 
5 + 3 q + 
 4 q^{2} + 3 q^{3} + q^{4} - 4 q^{5} - 6 q^{6} - 8 q^{7} - 8 q^{8} - 4 q^{9} - 
 2 q^{10} + 3 q^{11} + 4 q^{12} + 7 q^{13} + 5 q^{14} 
\\ & & + 4 q^{15} + q^{16} - 
 2 q^{18} - 2 q^{19} - 2 q^{20} - 2 q^{21} + q^{26} 
\\ f_5 &=& 
6 + 4 q + 6 q^{2} + 6 q^{3} + 
 6 q^{4} + 2 q^{5} - 3 q^{6} - 8 q^{7} - 12 q^{8} - 15 q^{9} - 16 q^{10} - 
 11 q^{11} - 8 q^{12} + 5 q^{14} 
\\ & & + 12 q^{15} + 14 q^{16} + 16 q^{17} + 12 q^{18} + 
 10 q^{19} + 4 q^{20} - q^{21} - 4 q^{22} - 7 q^{23} - 8 q^{24} - 8 q^{25} - 
 5 q^{26} - 4 q^{27} 
\\ & & - q^{28} + 2 q^{30} + 2 q^{31} + 2 q^{32} + 2 q^{33} + 
 2 q^{34} - q^{40} 
\\ f_6 &=& 
7 + 5 q + 8 q^{2} + 9 q^{3} + 11 q^{4} + 9 q^{5} + 7 q^{6} - 
 2 q^{7} - 7 q^{8} - 15 q^{9} - 22 q^{10} - 28 q^{11} - 30 q^{12} - 26 q^{13} 
\\ & & - 
 22 q^{14} - 11 q^{15} - 2 q^{16} + 13 q^{17} + 21 q^{18} + 33 q^{19} + 34 q^{20} + 
 36 q^{21} + 30 q^{22} + 25 q^{23} + 11 q^{24} + 3 q^{25} 
\\ & & - 8 q^{26} - 17 q^{27} - 
 22 q^{28} - 24 q^{29} - 24 q^{30} - 20 q^{31} - 14 q^{32} - 10 q^{33} - q^{34} + 
 2 q^{35} + 7 q^{36} + 8 q^{37} 
\\ & & + 11 q^{38} + 9 q^{39} + 8 q^{40} + 5 q^{41} + 
 4 q^{42} + q^{43} - 2 q^{45} - 2 q^{46} - 2 q^{47} - 2 q^{48} - 2 q^{49} - 
 2 q^{50} + q^{57}
\end{eqnarray*}
}
The general term of the sequence $f$ is given by:
\begin{equation}  
\lbl{eq.fnq} 
f_n(q)=\sum_{k=0}^n \frac{(q)_{n+k}}{(q)_{n-k} (q)_{k}}
\end{equation}
where the $q$-factorial is defined by
$$
(q)_n=\prod_{k=1}^n (1-q^k), \qquad (q)_0=1.
$$
The above expression implies via the Wilf-Zeilberger theorem that $f$
is $q$-holonomic; see \cite{WZ}. The {\tt qzeil.m} implementation 
of the Wilf-Zeilberger proof developed by \cite{PR1,PR2} gives that
$f$ is annihilated by the following operator:

\begin{eqnarray*}
P_f &=& (-1 + M q^2) (-1 + M q + M^2 q^2) L^2 
+
(-2 + M q + M q^2 + 2 M^2 q^2 + M^2 q^3 + 2 M^2 q^4 - 2 M^3 q^4 - 
\\ & & 
    2 M^3 q^5 - M^4 q^5 - M^4 q^6 - M^4 q^7 + M^5 q^7 + M^5 q^8 + 
    M^6 q^9) L 
+ 1 - M q^2 - M^2 q^4. 
\end{eqnarray*}
The 2-dimensional Newton polytope $\calN_2(P_f)$ and the Newton 
polygon $N(P_f)$ are given by:
$$
\psdraw{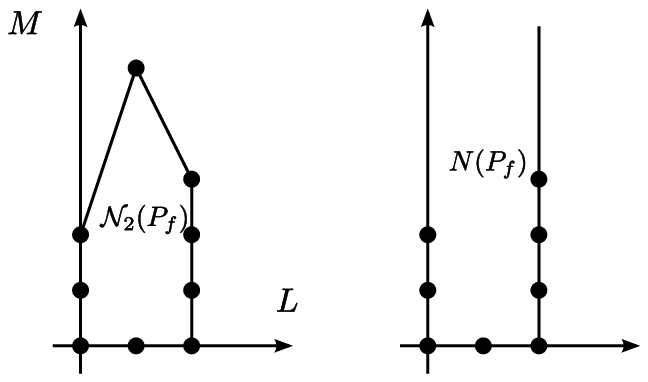}{2.5in}
$$
$P_f$ is a second order regular-singular $q$-difference operator.
Its Newton polygon $N(P_f)$ has only one slope $0$ with multiplicity $2$.
The edge polynomial of the $0$-slope is cyclotomic $(L-1)^2$ with two
equal roots (i.e., eigenvalues) $1$ and $1$.
The degree $\d_n$ and the leading term $\lt_n$ of $f_n(q)$ are given by 
$\d_n=0$ and $\lt_n=n+1$ for all $n$.

The characteristic curve $\ch_f$ (discussed in Section \ref{sec.invariants})
is reducible given by the zeros $(L,M) \in (\BC^*)^2$ of the polynomial
$$
(-1 + M + M^2) (-1 + 2 L - L^2 + L^2 M - 3 L M^2 + L M^3 + L M^4)=0
$$

The {\tt tropical.lib} program of \cite{Mk} computes 
the vertices of the tropical curve $\calT_f$ are:
   \begin{displaymath}
      (-1,-2),\;\;(3,-2),\;\;(0,-3/2),\;\;(1,-3/2),\;\;(0,-1)
   \end{displaymath}
The next figure is a drawing of the tropical curve $\calT_f$ and its 
multiplicities, where edges not labeled have multiplicity $1$.
  \vspace*{0.5cm}

   \begin{center}

    \begin{texdraw}
       \drawdim cm  \relunitscale 0.4 \arrowheadtype t:V
       \linewd 0.1  \lpatt (1 0)

       \setgray 0.6
       \relunitscale 3
       \move (-2 -1) \fcir f:0 r:0.06
       \move (-2 -1) \lvec (2 -1)
       \htext (0.5 -1.5){$2$}
       \move (-2 -1) \lvec (-1 -0.5)
       \move (-2 -1) \rlvec (-2.5 -0.83333333333333333333333333333333333333333333333333)
       \move (-2 -1) \rlvec (-1 0)
       \move (2 -1) \fcir f:0 r:0.06
       \move (2 -1) \lvec (0 -0.5)
       \move (2 -1) \rlvec (1 0)
       \htext (2.33 -1){$2$}
       \move (2 -1) \rlvec (2.5 -0.625)
       \move (-1 -0.5) \fcir f:0 r:0.06
       \move (-1 -0.5) \lvec (0 -0.5)
       \htext (0 -1){$2$}
       \move (-1 -0.5) \lvec (-1 0)
       \move (0 -0.5) \fcir f:0 r:0.06
       \move (0 -0.5) \lvec (-1 0)
       \move (-1 0) \fcir f:0 r:0.06
       \move (-1 0) \rlvec (-1 0)
       \htext (-1.33 0){$2$}
       \move (-1 0) \rlvec (0 1)
       \htext (-1 0.33){$2$}

        \move (-3 -2) \fcir f:0.8 r:0.03
        \move (-3 -1) \fcir f:0.8 r:0.03
        \move (-3 0) \fcir f:0.8 r:0.03
        \move (-3 1) \fcir f:0.8 r:0.03
        \move (-2 -2) \fcir f:0.8 r:0.03
        \move (-2 -1) \fcir f:0.8 r:0.03
        \move (-2 0) \fcir f:0.8 r:0.03
        \move (-2 1) \fcir f:0.8 r:0.03
        \move (-1 -2) \fcir f:0.8 r:0.03
        \move (-1 -1) \fcir f:0.8 r:0.03
        \move (-1 0) \fcir f:0.8 r:0.03
        \move (-1 1) \fcir f:0.8 r:0.03
        \move (0 -2) \fcir f:0.8 r:0.03
        \move (0 -1) \fcir f:0.8 r:0.03
        \move (0 0) \fcir f:0.8 r:0.03
        \move (0 1) \fcir f:0.8 r:0.03
        \move (1 -2) \fcir f:0.8 r:0.03
        \move (1 -1) \fcir f:0.8 r:0.03
        \move (1 0) \fcir f:0.8 r:0.03
        \move (1 1) \fcir f:0.8 r:0.03
        \move (2 -2) \fcir f:0.8 r:0.03
        \move (2 -1) \fcir f:0.8 r:0.03
        \move (2 0) \fcir f:0.8 r:0.03
        \move (2 1) \fcir f:0.8 r:0.03
        \move (3 -2) \fcir f:0.8 r:0.03
        \move (3 -1) \fcir f:0.8 r:0.03
        \move (3 0) \fcir f:0.8 r:0.03
        \move (3 1) \fcir f:0.8 r:0.03
                          
    \end{texdraw}\end{center}

   \vspace*{0.5cm}
   The Newton subdivision of the Newton polytope $\calN_2(P_f)$ is:
   \vspace*{0.5cm}

   \begin{center}
            
    \begin{texdraw}
       \drawdim cm  \relunitscale 0.5
       \linewd 0.05
        \move (1 6)        
        \lvec (2 3)
        \move (2 3)        
        \lvec (2 0)
        \move (2 0)        
        \lvec (0 0)
        \move (0 0)        
        \lvec (0 2)
        \move (0 2)        
        \lvec (1 6)

        \move (1 5)        
        \lvec (1 6)
        \move (1 4)        
        \lvec (1 5)
        \move (2 2)        
        \lvec (1 4)
        \move (1 4)        
        \lvec (0 0)
        \move (1 2)        
        \lvec (1 4)
        \move (2 2)        
        \lvec (1 2)
        \move (1 2)        
        \lvec (0 0)
        \move (0 0) \fcir f:0.6 r:0.1
        \move (0 1) \fcir f:0.6 r:0.1
        \move (0 2) \fcir f:0.6 r:0.1
        \move (0 3) \fcir f:0.6 r:0.1
        \move (0 4) \fcir f:0.6 r:0.1
        \move (0 5) \fcir f:0.6 r:0.1
        \move (0 6) \fcir f:0.6 r:0.1
        \move (1 0) \fcir f:0.6 r:0.1
        \move (1 1) \fcir f:0.6 r:0.1
        \move (1 2) \fcir f:0.6 r:0.1
        \move (1 3) \fcir f:0.6 r:0.1
        \move (1 4) \fcir f:0.6 r:0.1
        \move (1 5) \fcir f:0.6 r:0.1
        \move (1 6) \fcir f:0.6 r:0.1
        \move (2 0) \fcir f:0.6 r:0.1
        \move (2 1) \fcir f:0.6 r:0.1
        \move (2 2) \fcir f:0.6 r:0.1
        \move (2 3) \fcir f:0.6 r:0.1
        \move (2 4) \fcir f:0.6 r:0.1
        \move (2 5) \fcir f:0.6 r:0.1
        \move (2 6) \fcir f:0.6 r:0.1
       \move (1 6) 
       \fcir f:0 r:0.12
       \move (1 5) 
       \fcir f:0 r:0.12
       \move (2 3) 
       \fcir f:0 r:0.12
       \move (1 4) 
       \fcir f:0 r:0.12
       \move (2 2) 
       \fcir f:0 r:0.12
       \move (0 2) 
       \fcir f:0 r:0.12
       \move (0 1) 
       \fcir f:0 r:0.12
       \move (0 0) 
       \fcir f:0 r:0.12
       \move (1 2) 
       \fcir f:0 r:0.12
       \move (2 1) 
       \fcir f:0 r:0.12
       \move (2 0) 
       \fcir f:0 r:0.12
       \move (1 1) 
       \fcir f:0 r:0.12
       \move (1 0) 
       \fcir f:0 r:0.12
   \end{texdraw}
     
   \end{center}

To illustrate Theorem \ref{thm.1}, 
observe that $f_n(q)$ is a sequence of Laurent polynomials. 
The minimum (resp., maximum) degree $\d_n$ (resp., $\hat\d_n$) of $f_n(q)$
with respect to $q$ is given by:

\begin{equation*}
\d_n=0,  \qquad \hat\d_n=\frac{n(3n+1)}{2}
\end{equation*}
The coefficient $\lt_n$ (resp., $\widehat{\lt}_n$) of $q^{\d_n}$ (resp., 
$q^{\hat\d_n}$) in $f_n(q)$ is given by:

\begin{equation*}
\lt_n=n+1, \qquad \widehat{\lt}_n=(-1)^n
\end{equation*}
$\lt_n$ and $\widehat{\lt}_n$ are holonomic sequences that satisfy 
linear recursion relations with constant coefficients. 

To illustrate Theorem \ref{thm.4}, observe that there is a single $0$ slope
of length $2$ with edge polynomial $(L-1)^2$ and eigenvalue $1$ with 
multiplicity $2$. This is a resonant case. 
Theorem \ref{thm.4} dictates that we we substitute the WKB ansatz 
$$
\hat f_{n+j}(q)=\sum_{k=0}^\infty ( \phi_k(q) + (n+j) \psi_k(q)) q^{(n+j)k}
$$
in the recursion $P_f \hat f=0$, collect terms with respect to $M=q^k$
and with respect to $n$ and set them all equal to zero. It
follows that the vector $(\psi_k(q),\phi_k(q))$ satisfies a sixth order 
linear recursion relation:
\begin{eqnarray*}
\psi_k&=&\frac{1}{\left(1-q^k\right)^2}\left\{ q^{3+k} \psi_{-6+k}
q^{2+k} (1+q) \psi_{-5+k} + q^{1+k} \left(1+q+q^2\right) \psi_{-4+k}
+ q^{-2+k} \left(2 q^3+2 q^4-q^k\right) \psi_{-3+k} \right .
\\
& & \left.
-\left(-q^4-q^{2 k-2}+2 q^{k}+q^{1+k}+2 q^{2+k}+q^{-1+2 k}\right) \psi_{-2+k}
+\left(q^2+q^{2 k-1}-q^{k}-q^{1+k}+q^{2 k}\right) \psi_{-1+k} \right\} 
\\
\phi_k&=&\frac{1}{\left(1-q^k\right)^2}\left\{
-q^{3+k} \phi_{-6+k}
-q^{2+k} (1+q) \phi_{-5+k}
+q^{1+k} \left(1+q+q^2\right) \phi_{-4+k}
+q^{-2+k} \left(2 q^3+2 q^4-q^k\right) \phi_{-3+k} 
\right. \\
& & \left.
-\frac{\left(-q^6-q^{2 k}+2 q^{2+k}+q^{3+k}
+2 q^{4+k}+q^{1+2 k}\right) \phi_{-2+k}}{q^2}
+\frac{\left(q^3+q^{2 k}-q^{1+k}-q^{2+k}+q^{1+2 k}\right) \phi_{-1+k}}{q}
\right. \\
& & \left.
+\frac{q^{3+k} \left(1+q^k\right) \psi_{-6+k}}{\left(-1+q^k\right)}
+\frac{q^{2+k} (1+q) \left(1+q^k\right) \psi_{-5+k}}{\left(-1+q^k\right)}
-\frac{q^{1+k} \left(1+q+q^2\right) \left(1+q^k\right) \psi_{-4+k}}{\left(-1+q^k\right)}
\right. \\
& & \left.
-\frac{2 q^{-2+k} \left(q^3+q^4-q^k+q^{3+k}+q^{4+k}\right) \psi_{-3+k}}{\left(-1+q^k\right)}
\right. \\
& & \left.
+\frac{q^{-2+k} \left(2 q^2+q^3+2 q^4-2 q^6-2 q^k+2 q^{1+k}+2 q^{2+k}+q^{3+k}+2 q^{4+k}\right) \psi_{-2+k}}{\left(-1+q^k\right)}
\right. \\
& & \left.
+\frac{q^{-1+k} \left(q+q^2-2 q^3-2 q^k-q^{1+k}+q^{2+k}\right) \psi_{-1+k}}{\left(-1+q^k\right)} \right\}
\end{eqnarray*}
with an arbitrary initial condition $(\psi_0,\phi_0) \in \BQ[[q]]^2$. 
Note that $\psi_k$
the first equation above is a recursion relation for $\psi_k$ and the
second equation is a recursion for $\phi_k$ that also involves $\psi_{k'}$
for $k' < k$. This is a general feature of the WKB in the case of resonanse.
It follows that our particular solution \eqref{eq.fnq} of the 
$q$-difference equation $P_f f=0$ has the form:
\begin{equation}
\lbl{eq.fnexpansion}
f_n(q)=\sum_{k=0}^\infty ( \phi_k(q) + n \psi_k(q)) q^{n k} \in \BQ[[q]]
\end{equation}
where $(\psi_k,\phi_k)$ are defined the above recursion with suitable
initial conditions. Note that Equation \eqref{eq.fnexpansion} implies that
the coefficient $c_{m,n}=a_m + n b_m$ of $q^m$ in $f_n(q)$ is a linear 
function of $n$ for $n \geq m$. 
These values determine our initial conditions by:
$$
\phi_0(q)=\sum_{m=0}^\infty a_m q^m
$$
and
$$
\psi_0(q)=\sum_{m=0}^\infty b_m q^m
$$
In fact, a direct computation shows that the 
first 31 values of $a_m$ for $m=0, \dots, 30$ are given by:
$$
1, -1, -4, -9, -19, -33, -59, -93, -150, -226, -342, -494, -721, \
-1011, -1425, -1960, -2695, -3633, -4903, 
$$
$$
-6506, -8633, -11312, \
-14796, -19157, -24773, -31744, -40608, -51578, -65372, -82341, \
-103522
$$
and the first 31 values of $b_m$ for $m=0, \dots, 30$ are given by:
$$
1, 1, 2, 3, 5, 7, 11, 15, 22, 30, 42, 56, 77, 101, 135, 176, 231, \
297, 385, 490, 627, 792, 1002, 1255, 1575, 
$$
$$
1958, 2436, 3010, 3718, \
4565, 5604
$$
The reader may recognize (and also confirm by an explicit computation) that
$$
\psi_0(q)=\frac{1}{(q;q)_\infty}=\prod_{m=0}^\infty \frac{1}{1-q^m}
$$
An extra-credit problem is to give an explicit formula for the power series
$\phi_0(q)$.  Note finally that Equation \eqref{eq.fnq} implies that
the specialization of $f_n(q)$ at $q=1$ is given by:
$$
f_n(1)=1
$$
for all $n$, much like the case of the colored Jones polynomial of a knot,
\cite{GL}.

\subsection{Plan of the paper}
\lbl{sub.plan}

Theorem \ref{thm.1} follows from Theorems \ref{thm.2} and \ref{thm.4}. 

Theorem \ref{thm.2} follows by two reductions using a $q$-analogue of
Hensel's lemma, analogous to the case of linear differential operators.
 The first reduction factors an operator with arbitrary
Newton polygon into a product of operators with a Newton polygon with
a single slope. After a change of variables, we can assume that these
operators are regular-singular. The second reduction factors a 
regular-singular $q$-difference operator into a product of first
order regular-singular operators with eigenvalues of possibly equal
constant term. Thus, we may assume that the $q$-difference operator
is regular-singular with eigenvalues of equal constant term. In that
case, we can construct a basis of formal WKB solutions of the required
type. The above proof also factorizes a linear $q$-difference
operator into a product of first order $q$-difference operators,
giving a proof of Theorem \ref{thm.3}.

Theorem \ref{thm.4} follows from the Lech-Mahler-Skolem theorem
on the zeros of generalized exponential sums.

Finally, in Section \ref{sec.invariants}, which is independent of the results
of the paper, we discuss several invariants of $q$-holonomic sequences,
and compute them all in some simple examples.

\section{Proof of Theorem \ref{thm.2}}
\lbl{sec.thm2}

\subsection{Reduction to the case of a single slope}
\lbl{sub.monoslope}

Theorem \ref{thm.2} follows ideas similar to the case of linear differential
operators, presented for example in \cite[Sec.3]{vPS}. Rather than
develop the whole theory, we will highlight the differences
between the differential case and the $q$-difference case.

In the case of linear differential operators, the basic operators
$x$ and $x \pt_x$ satisfy the inhomogeneous commutation relation
\begin{equation}
\lbl{eq.xptx}
[x \pt_x,x]=x
\end{equation}
The right hand side of the above commutation relation leads 
to consider Newton polygons are convex hulls of translates of the 
{\em second quadrant} $\BR_- \times \BR^+$. For examples of such
polygons,  see \cite[Sec.3.3]{vPS}. As a result, the slopes of the Newton 
polygons of linear differential operators are always non-negative rational
numbers. On the other hand, in the $q$-difference case, the $q$-commutation 
relation \eqref{eq.qML}
preserves the $L$ and $M$ degree of a monomial in $M,L$.

In the case of linear differential operators of a single variable $x$, 
the WKB solutions involve power series in $x^{1/r}$ and polynomials in $\log x$.
In the case of linear $q$-difference operators, the WKB solutions involve 
power series in $q^n$ and polynomials in $n$. In addition, the convergence 
conditions are easier to deal with when $q=0$, and much harder when $q=1$.
Convergence when $q=1$ involves regularizing factorially divergent 
formal power series.

Consider a $q$-difference operator $P$ and its Newton polygon $N(P)$.
In this section, we will allow the coefficients $a_i(M,q)$ of $P$ from
Equation \eqref{eq.recP} to be arbitrary elements of the field 
$K\{\{M\}\}$. Recall that the Newton polygon $N(P)$
consists of finitely many segments and two rays. We will say that $P$
is {\em slim} if all its monomials lie in the bounded segments of its
Newton polygon. We will say that $P_1 > P_2$ if the boundary of $N(P_1)$
lies in the interior of $P_2$ or in the interior of the two rays of $P_2$.
We will say that an operator $P$ from Equation \eqref{eq.recP} is 
{\em monic} if $a_d(M,q)=1$. Recall the {\em Minkowski sum} 
$$
N_1+N_2=\{a+b|\,\, a \in N_1, \,\, b \in N_2\}
$$
of two polytopes in the plane. Here and below, we will always consider
the Minkowski sum of two convex polytopes that arise from linear $q$-difference
operators, i.e., polytopes that consist of two vertical rays and some
bounded segments.

The next lemma is a $q$-version of Hensel's lifting lemma.

\begin{lemma}
\lbl{lem.hensel1}
If $P$ is monic and slim, for every partition of its set of slopes in
two different sets $S_1$ and $S_2$ there exist unique slim and monic operators
$P_1$ and $P_2$ with slopes $S_1$ and $S_2$ respectively,
and an operator $R(P)$ so that
$$
P=P_1 P_2 + R(P), \qquad R(P) > P.
$$
\end{lemma}

\begin{proof}
This follows as in the differential case. The only difference is
that monomials $M^a L^b$ and $M^{a'} L^{b'}$ commute up to a power of $q$.
The next example will explain the proof. Suppose $P=a M + b L + c M^2 L
+ M^3 L^2$ is a monic slim operator with Newton polygon $N$ with two 
slopes, where $a,b,c \in K$. Now $N=N_1+N_2$ shown as follows:
$$
\psdraw{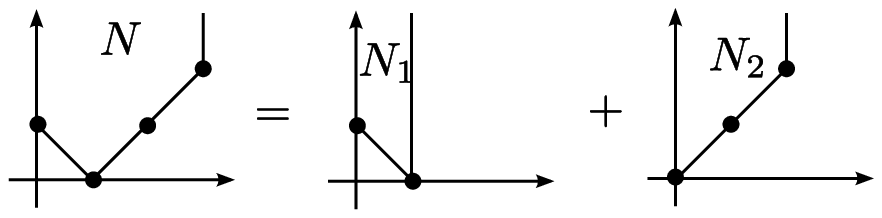}{4in}
$$
We want to find unique
$x_0,y_0,y_1,y_2 \in K$ so that 
$$
a M + b L + c M^2 L + M^3 L^2=(x_0 M + L)(y_0 + y_1 ML + y_2 M^2 L^2).
$$
Multiplying out, we obtain that
$$
a M + b L + c M^2 L + M^3 L^2=x_0 y_0 M + y_0 L + y_1 LML +R(P),
\qquad
R(P)=x_0 y_1 M^2 L + x_0 M^3L^2 + L M^2 L^2 > P.
$$
Thus, we obtain the system of equations
$$
x_0 y_0=a, \qquad y_0=b, \qquad q y_1=c.
$$
The above system can be solved uniquely in $K$ whether $q$ is present or not.
\end{proof}

\begin{proposition}
\lbl{prop.hensel1}
Suppose $P$ is monic and $N(P)=N_1+N_2$ where $N_1$ and $N_2$ have no
slope in common. Then there exist unique monic operators $P_1$ and $P_2$
with Newton polygons $N_1$ and $N_2$ 
such that $P=P_1 P_2$. In addition, we have:
$$
\calD/\calD P=\calD/\calD P_1 \oplus \calD/\calD P_2.
$$
\end{proposition} 

\begin{proof}
It follows verbatim from the proof of Theorem 3.48 in \cite{vPS}
using Lemma \ref{lem.hensel1}.
\end{proof}

\subsection{Reduction to the case of a single eigenvalue}
\lbl{sub.onelambda}

Suppose that $P$ is a linear $q$-difference operator with a single slope.
After a change of variables $f_n(q)=q^{n^2 \ga+ n\eta} g_n(q)$, we can assume that
the slope is zero, i.e. that $P$ is regular singular. In other words,
if $P$ is given by \eqref{eq.rec2}, then $a_d(0,q) a_0(0,q) \neq 0$.
Then, the slim part $S(P)$ of $P$ is given by
$$
S(P)=\sum_{i=0}^d a_i(0,q) L^i
$$
where $P-S(P)>P$. We may think of $S(P)$ as a polynomial in a variable
$L$ with coefficients in the field $K$. Suppose that $S(P)$ is monic
and we can factor $S(P)=A B$ as the product of two relatively prime monic
polynomials. 

\begin{proposition}
\lbl{prop.hensel2}
With the above assumptions there exist unique monic regular-singular
operators $P_1$ and $P_2$ such that $P=P_1 P_2$ and $S(P_1)=A$ and
$S(P_2)=B$. Moreover,
$$
\calD/\calD P=\calD/\calD P_1 \oplus \calD/\calD P_2.
$$
\end{proposition}

\begin{proof}
It follows verbatim from the proof of Proposition 3.50 in \cite{vPS}
using Lemma \ref{lem.hensel1}.
\end{proof}

\subsection{First order linear $q$-difference equation}
\lbl{sub.firstorder}

In this section we will prove Theorem \ref{thm.2} for the first order
linear $q$-difference equation
$$
f_{n+1}(q)-b(q^n,q) f_n(q)=0
$$
where $b(M,q) \in K((M))$. If $b(M,q)=M^{\ga} c(M,q)$
where $c(M,q) \in K[[M]]$ and  $c(0,q) \neq 0$
the substitution $f_n(q)=q^{n^2 \ga/2} c(0,q)^n g_n(q)$ gives the equation
\begin{equation}
\lbl{eq.order1}
g_{n+1}(q)-a(q^n,q) g_n(q)=0
\end{equation}
where $a(M,q)=q^{\ga/2} \frac{c(M,q)}{c(0,q)} \in K[[M]]$ satisfies
$a(0,q)=1$.
A solution to Equation \eqref{eq.order1} is
\begin{eqnarray*}
f_n(q) &=& a(q,q) a(q^2,q) \dots a(q^{n-1},q) \\
&=& \frac{\prod_{k=1}^\infty a(q^k,q)}{a(q^n,q) a(q^{n+1},q) \dots}
\end{eqnarray*}
Now
$$
a(M,q)=1-\sum_{k=1}^\infty M^k a_k(q)
$$
where $a_k(q) \in K$, implies that 
$$
\frac{1}{a(M,q)}=1+\sum_{l=1}^\infty 
\left(\sum_{k=1}^\infty M^k a_k(q)\right)^l \in u K[[u]]
$$
Consequently, 
\begin{equation}
\lbl{eq.fMq}
f(M,q)=\frac{1}{\prod_{i=0}^\infty a(Mq^i,q)} \in K[[M]]
\end{equation}
Let us now match this solution with a WKB sum. Consider
$$
f(u,q)=\sum_{k=0}^\infty \phi_k(q) u^k \in K[[u]]
$$
where $\phi_0=1$ and substitute into the equation 
$$
f(uq,q)-a(u,q) f(u,q)=0
$$
we obtain an equation in the ring $K[[u]]$. 
The vanishing of the coefficient of $u^k$ gives
$$
\phi_k=\frac{1}{1-q^k}\sum_{i=1}^k a_i \phi_{k-i}
$$
Thus the WKB solution matches exactly the solution \eqref{eq.fMq}
for the first-order regular singular equation \eqref{eq.order1}.

\subsection{Proof of Theorem \ref{thm.2}}
\lbl{sub.thm2.proof}

Propositions \ref{prop.hensel1} and \ref{prop.hensel2} and a possible
replacement of $q$ by $q^{1/r}$ reduce Theorem
\ref{thm.2} to the case of a monic regular-singular differential operator
of the form
\begin{equation}
\lbl{eq.dorder}
P=\prod_{i=1}^d (L-a_i(M,q))
\end{equation}
where $a_i(M,q) \in K[[M]]$ for all $i$ and $a(0,q)=1$. It suffices
to find a basis of solutions of the equation $Pf=0$ of the form
$$
f(u,q)=\sum_{k=0}^\infty \phi_k(n,q) u^k
$$
where $\phi_k(n,q) \in K[n]$ are polynomials of degree at most $d$.
We will use induction on $d$. When $d=1$, this is proven in Section
\ref{sub.firstorder}.
Let $Q=\prod_{i=2}^d (L-a_i(M,q))$. If we set $g_n(q)=f_{n+1}(q)-
a_1(q^n,q)f_n(q)$ then it $Pf=0$ implies that $Qg=0$. Conversely, the system
$Pf=0$ is equivalent to the system of equations
\begin{equation}
\lbl{eq.PQ}
\begin{cases} 
f(uq,q)-a_1(u,q)f(u,q)=g(u,q) \\
Qg(u,q)=0
\end{cases}
\end{equation}
Now, by induction we have
$$
g(u,q)=\sum_{k=0}^\infty \psi_k(n,q) u^k
$$
where $\psi_k(n,q) \in K[n]$ are polynomials of degree at most $d-1$.
Set $f(u,q)=\sum_{k=0}^\infty \phi_k(n,q) u^k$ 
where $\phi_k(n,q) \in K[n]$ are polynomials of degree at most $d$
in the first Equation \eqref{eq.PQ}. Both sides are elements of $K[[u]][n]$.
The vanishing of each coefficient of $u^k n^i$ computes the $\phi_k$ in terms 
of the $\psi_k$ by a rank 1 linear system of equations. It follows that
we can find a $d$-dimensional linear space of WKB series solutions of $Pf=0$.
This concludes the proof of Theorem \ref{thm.2}.

\subsection{The regular-singular non-resonant case}
\lbl{sub.regularsingular}

For completeness, we give a short proof of Theorem \ref{thm.2}
in the case of a {\em non-degenerate regular-singular operator}. 
Recall from Section \ref{sub.newton} that a regular-singular $q$-difference 
operator has the form
\begin{equation}
\lbl{eqw.nondeg}
P=\sum_{i=0}^d a_i(M,q) L^i
\end{equation}
where $a_i(M,q) \in K[[M^{1/r}]]$ for all $i$ and $a_0(0,q) a_d(0,q) \neq 0$.
The {\em characteristic polynomial} $\chi(x,M,q)$ of of a regular-singular
operator $P$ is given by
\begin{equation}
\lbl{eq.charpoly}
\chi(x,M,q):=\sum_{i=0}^d a_i(M,q) x^i
\end{equation} 
The {\em eigenvalues} of $P$ are the roots of the equation $\chi(x,M,q)=0$. 
Since $K$ is algebraically closed, Puiseux's theorem implies that the
eigenvalues of $P$ are elements of the field $K\{\{q\}\}$; see
\cite{Wa}. In other words, 
a regular-singular operator $P$ has $d$
eigenvalues $\l_i(M,q) \in K((M^{1/r'}))$ for some $r'$. 
After we replace $r$ by $r'$, we assume $r'=r$ below.
We say that a regular singular operator $P$ is {\em non-degenerate}
if its eigenvalues satisfy the {\em non-resonanse condition}
that $\l_j(0,q)/\l_i(0,q) \neq q^k$ for all $i \neq j$
and all $k$.

Suppose that $P$ is a non-degenerate regular-singular operator as above.
Consider Equation \eqref{eq.rec2}
\begin{equation}
\lbl{eq.rec22}
a_d(u,q) f(uq^d,q) + \dots + a_0(u,q) f(u,q)=0
\end{equation}
We will show that the general solution to \eqref{eq.rec22} is a WKB sum. 
Without loss of generality we assume $r=1$. We will construct a basis of 
solutions of the Equation \eqref{eq.rec22} of the form
\begin{equation}
\lbl{eq.fuq}
f(u,q)=\l(q)^n \sum_{k=0}^\infty \phi_k(q) u^k
\end{equation}
that satisfy $\phi_0(q)=1$. 
Here $n=\log u/\log q$. The point is that although
$f(u,q) \not\in K[[u]]$, the ratio
$f(uq^j,q)/\l(q)^n \in K[[u]]$ for every $j$. This follows from Equation
\eqref{eq.LMf}. Substitute \eqref{eq.fuq} into \eqref{eq.rec2}
and divide by $\l(q)^n$. The result is an equation in $K[[u]]$.
The vanishing of the coefficient of $u^0$ gives
$$
\sum_{j=0}^d a_j(0,q) \l(q)^j=0
$$
i.e., $x=\l(q)$ is a root of the equation $\chi(x,0,q)=0$.
The vanishing of the coefficient of $u^k$ for $k>0$ is
the condition
\begin{equation}
\lbl{eq.phik}
\phi_k(q) c_{0,k}(q) + \sum_{i=1}^k \phi_{k-i}(q) c_{i,k-i}(q)=0
\end{equation}
where 
$$
c_{i,j}(q)=(d/du)^i \chi(x,u,q)|_{(x,u)=(q^j\l(q),0)}.
$$
Using the initial condition $\phi_0=1$, it follows inductively that
$$
\psi_k(q):=\phi_k(q) \prod_{j=1}^k c_{0,j}(q)
$$ 
is a homogeneous polynomial of total degree 
$k$ in the variables $c_{i,j}(q)$, where the degree of $c_{i,j}(q)$ is $i+j$.
For example, we have:
\begin{eqnarray*}
\psi_1 &=& -c_{1,0}
\\
\psi_2 &=& c_{1,0} c_{1,1}-c_{0,1} c_{2,0}
\\
\psi_3 &=&
-c_{1,0} c_{1,1} c_{1,2}+c_{0,1} c_{1,2} c_{2,0}+c_{0,2} c_{1,0} c_{2,1}-c_{0,1} c_{0,2} c_{3,0}
\\
\psi_4 &=& 
c_{1,0} c_{1,1} c_{1,2} c_{1,3}-c_{0,1} c_{1,2} c_{1,3} c_{2,0}-c_{0,2} c_{1,0} c_{1,3} c_{2,1}-c_{0,3} c_{1,0} c_{1,1} c_{2,2}+c_{0,1} c_{0,3} c_{2,0} c_{2,2}
\\
& & +c_{0,1} c_{0,2} c_{1,3} c_{3,0}
+c_{0,2} c_{0,3} c_{1,0} c_{3,1}-c_{0,1} c_{0,2} c_{0,3} c_{4,0}
\\
\psi_5 &=& 
-c_{1,0} c_{1,1} c_{1,2} c_{1,3} c_{1,4}+c_{0,1} c_{1,2} c_{1,3} c_{1,4} c_{2,0}+c_{0,2} c_{1,0} c_{1,3} c_{1,4} c_{2,1}+c_{0,3} c_{1,0} c_{1,1} c_{1,4} c_{2,2}
\\ & &
-c_{0,1} c_{0,3} c_{1,4} c_{2,0} c_{2,2}+c_{0,4} c_{1,0} c_{1,1} c_{1,2} c_{2,3}-c_{0,1} c_{0,4} c_{1,2} c_{2,0} c_{2,3}-c_{0,2} c_{0,4} c_{1,0} c_{2,1} c_{2,3}
\\ & &
-c_{0,1} c_{0,2} c_{1,3} c_{1,4} c_{3,0}+c_{0,1} c_{0,2} c_{0,4} c_{2,3} c_{3,0}-c_{0,2} c_{0,3} c_{1,0} c_{1,4} c_{3,1}-c_{0,3} c_{0,4} c_{1,0} c_{1,1} c_{3,2}
\\ & &
+c_{0,1} c_{0,3} c_{0,4} c_{2,0} c_{3,2}+c_{0,1} c_{0,2} c_{0,3} c_{1,4} c_{4,0}+c_{0,2} c_{0,3} c_{0,4} c_{1,0} c_{4,1}-c_{0,1} c_{0,2} c_{0,3} c_{0,4} c_{5,0}
\end{eqnarray*}

Note that
$c_{0,k}=\chi(q^k\l(q),0,q) \neq 0$ for all $k$ due to the non-degeneracy of 
\eqref{eq.rec22}. It remains to show that there exists $c$ such that
$\d(\phi_k(q)) \geq c k $ for all $k$. To prove this, observe that
$$
c_{0,k}(q)=\sum_{j=0}^d a_j(0,q) (q^k \l(q))^j
$$
satisfies that 
$$
\d(c_{0,k})=\min_{0 \leq j \leq d}\{\d(a_j(0,q)\l(q)^j )+k j \} 
$$
is attained at least twice. Since $a_0(0,q) \neq 0$, it follows that for large
$k$, the above minimum is attained at $j=0$. In other words, we have
$$
\d(c_{0,k})=c:=\d(a_0(0,q)).
$$
Thus, 
$$
\d(\prod_{i=1}^k c_{0,i}(q)) \geq ck+c'
$$
for all $k$ large enough. On the other hand, $\d(c_{i,j}(q)) \geq 0$
for all $i,j$. The above formulas imply 
that $\d(\phi_k(q)) \geq c'k$ for all $k$.
Thus, for every eigenvalue $\l(0,q)$ of $P$ there is a unique WKB solution
\eqref{eq.fuq} of Equation \eqref{eq.rec22}. Since 
there are $d$ distinct 
eigenvalues, the constructed solutions form a basis for the vector
space of solutions of \eqref{eq.rec22}. This gives a proof of Theorem 
\ref{thm.2} in the non-degenerate regular-singular case.

\begin{remark}
\lbl{rem.compareODE}
In the case of regular-singular linear differential equations, expressions
of the form \eqref{eq.phik} appear. However, one needs to 
estimate the numerator of those expressions, as well as the denominator.
See for example, \cite{GG} and references therein. 
This explains why the WKB theory of $q$-difference equations at $q=0$
is much simpler than the corresponding theory at $q=1$.
\end{remark}

\section{Proof of Theorem \ref{thm.4}}
\lbl{sec.thm3}

\subsection{Generalized power sums}
\lbl{sub.gps}

An important special case of Theorem \ref{thm.1} is the case of a linear
recursion with constant coefficients. In this rather trivial case, for
every $n$, $f_n(q)$ is a constant function of $q$, so the degree is easy
to compute. 

Generalized power sums play a key role to the SML theorem. For a detailed
discussion, see \cite{vP} and also \cite{EvPSW}. Recall that a
{\em generalized power sum} $a_n$ for $n=0,1,2,\dots$ is an expression
of the form
\begin{equation}
\lbl{eq.gps}
a_n=\sum_{i=1}^m A_i(n) \a_i^n
\end{equation}
with roots $\a_i$, $1 \leq i \leq m$, distinct nonzero quantities,
and coefficients $A_i(n)$ polynomials respectively of degree $n(i)-1$ for
positive integers $n(i)$, $1 \leq i \leq m$. The generalized power sum $a_n$
is said to have order
$$
d=\sum_{i=1}^m n(i)
$$
and satisfies a linear recursion with constant coefficients of the form
$$
a_{n+d}=s_1 a_{n+d-1} + \dots + s_d a_n
$$
where 
$$
s(x)=\prod_{i=1}^m (1-\a_i x)^{n(i)}=1-s_1x-\dots s_d x^d.
$$
It is well known that a sequence satisfies a linear recursion with constant
coefficients if and only if it is a generalized power sum. The LMS theorem
concerns the zeros of a generalized power sum.

\begin{theorem}
\lbl{thm.SML}\cite{Sk,Ma,Le}
The zero set of a generalized power sum is the union of a finite set and
a finite set of arithmetic progressions.
\end{theorem}

\subsection{Proof of theorem \ref{thm.4}}
\lbl{sub.thm3}

Consider a formal WKB series 
$$
f_{\tau,n}(q)=q^{n^2 \ga + \eta n} \mu^n \sum_{j,k \geq 0}
c_{\tau,j,k}(n)  q^{\frac{j+nk}{r}} 
$$
where $\eta=\d(\l(q))$, $\mu=\lt(\l(q)) \in \overline{\BQ}$
and $c_{\tau,j,k}(n)$ are generalized power sums with $r$th roots of $1$. 
Fix a sequence $f_n(q)$ given by a finite $\BQ((q^{1/r}))$-linear combination
of formal WKB series $f_{\tau_i}(u,q)$. It follows that

\begin{equation}
\lbl{eq.eqsol}
f_n(q)=\sum_{i \in I} c_i(q) f_{\tau_i,n}(q)=
\sum_{i \in I} q^{n^2 \ga_i + \eta_i n} \mu_i^n \sum_{j,k \geq 0}
c_{i,j,k}(n)  q^{\frac{j+nk}{r}}
\end{equation}
where $I$ is a finite set, $c_i(q) \in \overline{\BQ}((q^{1/r}))$ are 
non-zero and $(\tau_i,\l_i)$ are pairwise distinct. 
Consider the subset $I'$ of $I$ where the minimum
$$
\min_{i \in I}\{ \ga_i\}
$$
is achieved. 

{\bf Case 1}. If $I'$ consists of a single element $i_0$ and $c_{i_0,j_0,k_0}(n)$
is not identically zero, then let us concentrate on the following part of
$f_n(q)$:
$$
f_n(q)=q^{n^2 \ga_{i_0} + \eta_{i_0} n} \mu_{i_0}^n c_{i_0,j_0,k_0}(n)
q^{\frac{j_0+n k_0}{r}} + \dots 
$$
The LMS theorem concludes that there is a finite set of full arithmetic 
progressions that cover the set of natural numbers with finite complement,
such that the restriction of the degree of  $f_n(q)$ to each full arithmetic
progression (with the exception of a finite set) is given 
by $n^2 \ga_{i_0} + \eta_{i_0} n+\frac{j_0+n k_0}{r}$ and 
the leading term is given by $\mu_{i_0}^n c_{i_0,j_0,k_0}(n)$. When $c_{i_0,j_0,k_0}(n)$
vanishes consider the next term in the series \eqref{eq.eqsol} and
repeat the above argument. After finitely many repetitions the process will
stop since $f_{\tau_i,n}(q)$ are $\overline{\BQ}((q^{1/r}))$-linearly
independent.

{\bf Case 2}. If $I'$ consists of more than one element, without loss
of generality, assume that $c_{i,0,0}(n)$ are not identically zero
for $i \in I'$. Consider the subset $I''$ of $I'$ where the 
maximum
$$
\max_{i \in I'}\{ \eta_i\}
$$
is achieved. 

{\bf Case 2.1}. $I''$ consists of a single element $i_0$. Then,
it follows that 
$$
f_n(q)=q^{n^2 \ga_{i_0} + \eta_{i_0} n} \mu_{i_0}^n c_{i_0,0,0}(n)
+ \dots 
$$
and the proof proceeds as in Case 1.

{\bf Case 2.2}.  $I''$ consists of more than a single element $i_0$.
In that case, it follows that
$$
f_n(q)=q^{n^2 \ga_{i_0} + \eta_{i_0} n} \sum_{i \in I''} \mu_{i}^n c_{i,0,0}(n)
+ \dots 
$$
The leading term of $f_n(q)$ is a generalized power sum 
and the proof proceeds as in Case 1.

\section{Invariants of $q$-holonomic sequences}
\lbl{sec.invariants}

\subsection{Synopsis of invariants}
\lbl{sub.synopsis}

In this section, which is independent of the results of our paper,
we summarize various invariants of a a $q$-holonomic sequence.
Some of these invariants were announced in \cite{Ga1,Ga2,Ga3}.
In this section $f$ denotes a $q$-holonomic sequence  
$f_n(q) \in \BQ(q)$ that satisfies Equation \eqref{eq.rec}
where $a_i(M,q) \in \BQ(M,q)$. Here is a summary of invariants of $f$.

\begin{itemize}
\item[(a)]
The annihilator polynomial $P_f$ 
\item[(b)]
The 2 and 3-dimensional Newton polytopes.
\item[(c)]
The characteristic curve $\ch_f$.
\item[(d)]
The tropical curve $\calT_f$.
\item[(e)]
The degree and leading term of $f$.
\end{itemize}
As was explained in \cite{Ga3}, these invariants fit well together
and read information of the $q$-holonomic sequence $f$. For completeness,
we will present these ideas here, too.

\subsection{The annihilating polynomial of a $q$-holonomic sequence}
\lbl{sub.annih}

In this section we discuss the annihilating polynomial of a $q$-holonomic
sequence $f$. Informally, it is the minimal order recursion relation 
(homogeneous or not) that $f$ satisfies. This was studied in \cite{Ga1} and
\cite{Ga3}.
 We will call an operator $P$ given by \eqref{eq.recP} with $r=1$ 
{\em reduced} if $a_i(M,q) \in \BZ[M,q]$ are coprime.
Let 
$$
\calD'=\BQ(M,q)\la L\ra/(LM-qML)
$$ 
denote the {\em localized} $q$-Weyl algebra.

First, we define the homogeneous annihilator $P_f$ of $f$. Consider
the $\calD'$-module (nonzero, since $f$ is $q$-holonomic)
\begin{equation}
\lbl{eq.Mf}
M_f=\{ P \in \calD'  \,\, | P f=0 \}
\end{equation}
and its unique monic generator $P' \in \calD'$. After multiplying by the common
denominator of the coefficients of $P'$ with respect to the powers of $L$,
we obtain a unique reduced generator $P_f$.

Next, we define the non-homogeneous annihilator $(P_f,b_f)$ of $f$.
Consider the $\calD'$-module 
\begin{equation}
\lbl{eq.Mhnf}
M^{nh}_f=\{ (P,b) \in \calD' \times \BQ(M,q) \,\, | P f=b \}
\end{equation}
$M^{nh}_f$ has a unique generator of the form $(P',\e)$ where $P' \in \calD'$ 
is monic (with respect to $L$) and $\e \in \{0,1\}$. 
After multiplying by the common
denominator of the coefficients of $P'$ with respect to the powers of $L$,
we obtain a unique generator of the form $(P^{nh}_f,b_f)$ where $P^{nh}_f$ 
is reduced and $b_f \in \BZ[M,q]$.

\begin{definition}
\lbl{def.nhom}
We call $P_f$ and $(P^{nh}_f,b_f)$ the {\em homogeneous} and the 
{\em non-homogeneous} annihilator of the $q$-holonomic sequence $f$, 
respectively.
\end{definition}

In \cite[Lem.1.4]{Ga3} we show how $P_f$ determines $(P^{nh}_f,b_f)$
and vice-versa.

\subsection{The characteristic variety of a $q$-holonomic sequence}
\lbl{sub.cv}

The {\em characteristic curve} of a $q$-holonomic sequence $f$ is
the zero set
\begin{equation}
\lbl{eq.cv}
\ch_f=\{(L,M) \in (\BC^*)^2 \,\, | P_f(L,M,1)=0 \}
\end{equation}
In case $f$ is the colored Jones polynomial of a knot $K$, the AJ Conjecture
of \cite{Ga2} identifies the characteristic curve with a geometric
knot invariant, namely the moduli space of $\SL(2,\BC)$ representations
of the knot complement.

\subsection{The Newton polytope of a $q$-holonomic sequence}
\lbl{sub.np}

In this section we consider the 3-dimensional and the 2-dimensional
Newton polytope of a $q$-holonomic sequence $f$. 
Let us write the annihilating polynomial $P_f$ of $f$
in terms of monomials:
\begin{equation}
\lbl{eq.Pijk}
P_f=\sum_{(i,j,k) \in \calA} a_{i,j,k} q^k M^j L^i.
\end{equation}
where $a_{i,j,k} \in \BQ$ and the sum is over a finite subset $\calA$ of $\BN^3$ 
which depends on $f$.

\begin{definition}
\rm{(a)}
The 3-{\em dimensional Newton polytope} $\calN_3(P_f)$ of $f$ is 
the convex hull $\calN_3(P_f)$ of $\calA$.
\newline
\rm{(b)}
The 2-{\em dimensional Newton polytope} $\calN_2(P_f)$ of $f$ is
in $\BR^3$ is the projection of $\calN_3(P_f)$
on $\BR^2$ under the map $(x,y,z) \mapsto (x,y)$.
\end{definition}
The {\em width} of $\calN_3(P_f)$ with respect to $L$ is
the degree of a minimal order recursion relation for $f$, and it is
a measure of the complexity of $f$. Likewise, the width of $\calN_3(P_f)$
with respect to $M$ is a measure of the complexity of the coefficients 
of a recursion relation of $f$. Knowing (or guessing) $\calN_3(P_f)$
given some terms in $f$ has applications in Quantum Topology.

Generically one expects that the Newton polygon of the polynomial 
$P_f(x,y,1)$ coincides with $\calN_2(P_f)$.

\begin{lemma}
\lbl{lem.N2N1}
The lower convex hull of $\calN_2(P_f)$ coincides with the Newton
polygon $N(P_f)$ of $P_f$ from Section \ref{sub.newton}.
\end{lemma}

\subsection{The tropical curve of a $q$-holonomic sequence}
\lbl{sub.tropical}

In this section, we assign a tropical curve $\calT_f$ to a $q$-holonomic
sequence $f$, following \cite{Ga3}. For a leisure introduction to
tropical geometry, see \cite{RGST,SS,Sm,Ga3} and references therein.
The main idea is to think of an operator in $M,L$ with coefficients in $\BQ(q)$
as a $q$-parameter family of polynomials in two {\em commuting} variables.
More precisely, the coefficients of the annihilating 
polynomial $P_f$ given by \eqref{eq.Pijk} define a function
\begin{equation}
\lbl{eq.Ftrop}
F: \BR^2 \longto \BR, \qquad
F(x,y)=\min_{(i,j,k) \in \calA}\{i x + j y + k\}
\end{equation}
$F$ is a piecewise linear convex function. The locus of the points in $\BR^2$
where $F$ is not differentiable is the {\em tropical curve} $\calT_f$
of $f$. By definition, $\calT_f$ is the locus of points $(x,y) \in \BR^2$
where the minimum in \eqref{eq.Ftrop} is achieved {\em at least twice}.
It is well-known that $\calT_f$ consists of a finite collection of line 
segments with rational vertices, and a finite collection of rays with
rational slopes, together with a set of multiplicities
that satisfy a {\em balancing condition}. Abstractly,
a tropical curve is a {\em balanced rational graph}, and vice-versa; see 
\cite[Thm.3.6]{RGST}. 

There is a duality between Newton polytopes and tropical curves. Indeed,
the projection of the lower hull of $\calN_3(P_f)$ gives a Newton subdivision
of $\calN_2(P_f)$. The tropical curve $\calT_f$ is {\em dual} to the Newton 
subdivision of $\calN_2(P_f)$. For drawings of tropical curves
of $q$-holonomic sequences of geometric origin, see \cite{Ga3}.

\subsection{A tropical equation for the degree of a $q$-holonomic sequence}
\lbl{sub.degholo}

In this section we explain the relation between the degree $\d(n)=\d(f_n(q))$
of a $q$-holonomic sequence $f$ and its tropical curve $\calT_f$,
following \cite{Ga3}. Since $f$ is annihilated by $P_f$ (given by 
\eqref{eq.Pijk}), it follows that for all natural numbers $n$ we have:
$$
\sum_{(i,j,k) \in \calA} a_{i,j,k} q^{k+jn} f_{n+i}(q)=0.
$$
Divide both sides by $f_n(q)$ and look at the degree with respect to $q$.
It follows that for all $n$, $\d(n)$ satisfies the following
{\em tropical equation}: the minimum
\begin{equation}
\lbl{eq.d2}
\min_{(i,j,k) \in \calA} \{\d(n+i)-\d(n) +jn+k \}
\end{equation}
is achieved at least twice. 
Although \eqref{eq.d2} may have non quasi-polynomial solutions, 
Theorem \ref{thm.1} suggests to look for quadratic quasi-polynomial
solutions of \eqref{eq.d2}.
It turns out that those solutions $\d(n)$ are determined by the tropical
curve $\calT_f$ as follows. For $n$ in an arithmetic progression minus
finitely many values, we have:

\begin{equation}
\lbl{eq.dc}
\d(n)=\frac{c_2}{2} n^2 + c_1 n + c_0
\end{equation}

\begin{proposition}
\lbl{prop.c1c2}
$(c_1,c_2)$ satisfy the system of equations
\begin{eqnarray}
\lbl{eq.c12a}
c_2 i + j &=& c_2 i' + j' \\
\lbl{eq.c12b}
\frac{1}{2} c_2 i^2 + c_1 i + k &=& \frac{1}{2} c_2 i'^2 + c_1 i' + k'
\end{eqnarray}
for two distinct points $(i,j,k)$ and $(i',j',k')$ of $\calA$. 
\end{proposition}

\begin{proof}
For $n$ in an arithmetic progression minus finitely many values we have:
the minimum
\begin{equation}
\lbl{eq.d3}
\min_{(i,j,k) \in \calA} \{ n(c_2 i + j) + \frac{1}{2} c_2 i^2 + c_1 i +k \}
\end{equation}
is achieved at least twice. 
It follows that there are two distinct points $(i,j,k)$ and $(i',j',k')$
of $\calA$ such that Equations \eqref{eq.c12a} and \eqref{eq.c12b} hold.
Equation \eqref{eq.c12a} involves only $(i,j)$ and $(i',j')$
and says that $c_2$ is the negative of a slope of $N(P_f)$. Likewise
Equation \eqref{eq.c12a} involves only $(i,k)$ and $(i',k')$.
\end{proof}

\begin{remark}
\lbl{rem.c12a}
It would be interesting to study the solutions to the tropical equation
\eqref{eq.d2} independent of differential Galois theory.
\end{remark}

\begin{remark}
\lbl{rem.c12b}
Each Equation \eqref{eq.c12a} and \eqref{eq.c12b} describes a tropical
curve in $\BR^2$. However, the system of Equations   
\eqref{eq.c12a} and \eqref{eq.c12b} is {\em not} the intersection of the two
tropical curves since the two planes $\BR^2$ are different projections
of $\BR^3$.
\end{remark}

\subsection{Future directions}
\lbl{sub.future}

$q$-holonomic sequences of several variables 
$f_{n_1,\dots,n_r}(q) \in \BZ[q^{\pm 1}]$ also appear in Quantum Topology.
For example the colored Jones function of a 2-component link, of the
$\mathfrak{sl}_3$-colored Jones polynomial of a knot is a $q$-holonomic
sequence of two variables. Suitably formulated, Theorem \ref{thm.1} should 
extend to $q$-holonomic sequences of many variables. This will be explained
in a future publication.

\subsection{Acknowledgment}
The idea of the present paper was conceived during the New York
Conference on {\em Interactions between Hyperbolic Geometry, 
Quantum Topology and Number Theory} in New York in the summer of 2009.
An early version of the present paper appeared in the New Zealand Conference
on {\em Topological Quantum Field Theory and Knot Homology Theory}
in January 2010 and a finished version appeared in the Conference in
Rutgers in honor of D. Zeilberger's 60th birthday.  
The author wishes to thank the organizers of the New York Conference, 
A. Champanerkar, O. Dasbach, E. Kalfagianni, I. Kofman, W. Neumann and 
N. Stoltzfus, the New Zealand Conference,  R. Fenn, D. Gauld
and V. Jones and the Rutgers Conference, for their hospitality.
In addition, the author wishes to thank E. Croot, T. Dimofte, D. Zagier, 
D. Zeilberger and J. Yu for stimulating conversations.

\bibliographystyle{hamsalpha}\bibliography{biblio}
\end{document}